\documentclass[12pt]{amsart}
\usepackage[left=1in,top=1.2in,bottom=1in]{geometry}
\usepackage{tikz, url, color, amsthm, amssymb, comment, enumerate, graphicx,hyperref}
\usepackage{cleveref}
\usepackage[nodisplayskipstretch]{setspace}

\newcommand{\Or}{\mathrm{O}}
\newcommand{\SO}{\mathrm{SO}}

\newcommand{\C}{\mathbb{C}}
\renewcommand{\H}{\mathbb{H}}
\newcommand{\N}{\mathbb{N}}
\newcommand{\Q}{\mathbb{Q}}
\newcommand{\Z}{\mathbb{Z}}
\newcommand{\R}{\mathbb{R}}
\newcommand{\PSL}{\mathrm{PSL}}

\newcommand{\GL}{\mathrm{GL}}

\newcommand{\floor}[1]{\lfloor {#1} \rfloor}
\newcommand{\Isom}{\mathrm{Isom}}

\theoremstyle{plain}
\newtheorem{theorem}{Theorem}[section]
\newtheorem{lemma}[theorem]{Lemma}

\newtheorem{conjecture}[theorem]{Conjecture}

\newtheorem{remark}[theorem]{Remark}
\newtheorem{definition}[theorem]{Definition}
\newtheorem{mainthm}{Theorem}

\newtheorem{maincor}[mainthm]{Corollary}

\title{Counting Salem numbers arising from arithmetic hyperbolic orbifolds}

\author[M. Chu]{Michelle Chu}
\address{School of Mathematics\\ University of Minnesota \\ 
127 Vincent Hall 206 Church St. SE\\ Minneapolis, MN 55455, United States}
\email{mchu@umn.edu}

\author[P. G. P. Murillo]{Plinio G. P. Murillo}
\address{Instituto de Matemática e Estatística\\ Universidade Federal Fluminense \\
Rua Professor Marcos Waldemar de Freitas Reis,  s/n, Campus do Gragoatá, São Domingos, Niterói, RJ.
24210-201, Brazil}
\email{pliniom@id.uff.br}

\author[O. Romero]{Otto Romero}
\address{Centro de Investigaci\'on en Matem\'aticas AC, Jalisco s/n, Valenciana,
Guanajuato, Gto. 36023 M\'exico}
\email{otto.romero@cimat.mx}

\author[L. Thompson]{Lola Thompson}
\address{Universiteit Utrecht, Hans Freudenthalgebouw\\
Budapestlaan 6\\
3584 CD Utrecht, The Netherlands}
\email{l.thompson@uu.nl}

\begin{document}

\begin{abstract}
The relationship between Salem numbers and short geodesics has been fruitful in quantitative studies of arithmetic hyperbolic orbifolds, particularly in dimensions 2 and 3. In this article, we push these connections even further. The primary goals are: (1) to bound the proportion of Salem numbers of degree up to $n+1$ in the commensurability class of classical arithmetic lattices in any odd dimension $n$; (2) to improve lower bounds for the strong exponential growth of averages of multiplicities in the geodesic length spectrum of non-compact arithmetic orbifolds. In order to accomplish these goals, we bound, for a fixed square-free integer $D$, the count of Salem numbers with minimal polynomial $f$ satisfying $f(1)f(-1)=-D$ in $\Q^{\times}/\Q^{\times 2}$. To do this, we make use of results on the distribution of Salem numbers, as well as classical methods for counting Pythagorean triples and Gauss' lattice-counting argument. To this end, we give a generalization of the count of Pythagorean triples and provide an elementary proof which may be of independent interest.
\end{abstract}


\maketitle

\section{Introduction}\label{s:intro}

A \emph{Salem number} is a real algebraic integer $\lambda>1$ such that all of its Galois conjugates except $\lambda^{-1}$ have absolute value equal to 1. The minimal polynomial $f(x)$ of a Salem number is necessarily symmetric and of even degree. In addition to the roots $\lambda$ and $\lambda^{-1}$, all other roots appear in complex pairs along the unit circle. 
Salem numbers arise in many areas of mathematics, and their study is closely related to the elusive Lehmer's problem on Mahler measure. 

The distribution of Salem numbers was studied by G{\"o}tze and Gusakova in \cite{GG20}, where they obtained an asymptotic count of the Salem numbers of fixed degree with bounded Mahler measure. In this article, we restrict this to count the Salem numbers of fixed degree with bounded Mahler measure whose minimal polynomials satisfy a certain geometrically-motivated condition.

\begin{mainthm} \label{ThmSalemCount}
Let $m$ be a positive integer and $D$ a square-free positive integer. The number of Salem polynomials $f$ of degree $2m$ for which $f(1)f(-1)\equiv -D$ in $\Q^{\times}/\Q^{\times 2}$ with Mahler measure $\leq Q$ is bounded above by
\[
    \frac{\kappa(m,D)}{\pi\sqrt{D}}Q^{m-1}\log Q + O(Q^{m-1})
\]
where $\kappa(m,D)$ is a constant depending only on $D$ and $m$.
\end{mainthm}

The condition $f(1)f(-1)\equiv -D$ has applications to the study of arithmetic hyperbolic orbifolds. Indeed, Bayer-Fluckiger \cite{Bayer-Fluckinger} stated this condition as a necessary condition for the existence of an isometry in a quadratic space with characteristic polynomial $f$. 
It is known that Salem numbers are closely related to the lengths of geodesics in arithmetic hyperbolic orbifolds.
Recent work of Emery, Ratcliffe, and Tschantz made this relationship more explicit.

\begin{theorem}\cite[Theorem 1.1]{EmeryRatcliffeTschantz}\label{th:ERT 1.1}
    Let $\Gamma\subset\mathrm{Isom}(\H^n)$ be a classical arithmetic lattice of the first type defined over a totally real number field $K$. Let $\gamma$ be a hyperbolic element of $\Gamma$, let $\ell(\gamma)$ be the translation length of $\gamma,$ and let $\lambda=e^{\ell(\gamma)}$. Then 
    \begin{enumerate}
        \item $\lambda$ is a Salem number, and $\deg_{K}(\lambda)=\deg_{\Q}(\lambda)\leq n+1$

        \item $K$ is a subfield of $\Q(\lambda +\lambda^{-1})$, and $[\Q(\lambda+\lambda^{-1}):K]=\deg_K(\gamma)/2$.  
    \end{enumerate}
\end{theorem}

The theorem above implies that a Salem number $\lambda$ can only be realized by an isometry in a classical arithmetic lattice $\Gamma\subset\mathrm{Isom}(\H^n)$ defined over the rational numbers $\Q$ if $\deg(\lambda)\leq n+1$. We consider the following question. 
\begin{quote}\itshape
    What is the proportion of Salem numbers of degree $2m\leq n+1$ realized by a classical arithmetic hyperbolic lattice of the first type of dimension $n$ defined over $\Q$?
\end{quote}

This question was addressed in \cite{BLMT22} in the case where $m=2$, $n=3$, and $\Gamma$ is a Bianchi group. We obtain the following theorem, which bounds the proportion of Salem numbers of degree up to $n+1$ in the commensurability class of classical arithmetic hyperbolic lattices in odd dimensions $n$. Note that this theorem applies to all non-cocompact arithmetic lattices in odd dimensions but also to some cocompact lattices in dimension 3.

\begin{mainthm} \label{ThmSalemProportion}
 Let $q$ be a rational quadratic form with signature $(n,1)$ with $n$ odd and reduced determinant $-D$. The number of Salem numbers $\lambda\leq Q$ realized in the commensurability class of $\SO(q,\Z)$ is bounded above by 
    \[ \frac{\kappa(m,D)}{\pi\sqrt{D}}Q^{m-1}\log Q + O(Q^{m-1})
    \]
    where $2m=n+1$ and the positive constants $D$ and $\kappa(m,D)$ depend only on $q$.
\end{mainthm}

As a consequence, the aforementioned connection between Salem numbers and closed geodesics of arithmetic hyperbolic orbifolds of the first type allow us to prove the following.

\begin{maincor} \label{ThmLengths}
Let $n$ be odd and let $\Gamma$ be a classical arithmetic lattice in $\mathrm{Isom}(\H^n)$ commensurable with some $\Or(q,\Z)$ for an admissible quadratic form $q$ over $\Q$. Suppose $q$ has reduced determinant $-D$. Let $m=\frac{n+1}{2}$. The number of distinct lengths $\leq L$ of hyperbolic isometries which are elements of classical arithmetic lattices in the commensurability class of $\Gamma$ is bounded above by
\[
    \frac{\kappa(m,D)}{\pi\sqrt{D}}e^{(m-1)L}L + O(e^{(m-1)L})
\]
where $\kappa(m,D)$ is a constant depending only on $D$ and $m$.
\end{maincor}

The count above applies to arithmetic lattices of the first type defined over $\Q$ which are \emph{classical}, that is, they are contained in some $\SO(q,\Q)$. In odd dimension, not all arithmetic lattices $\Gamma$ of the first type are classical; however, the subgroup $\Gamma^{(2)}$ generated by squares of elements in $\Gamma$ is always classical (see \cite[Lemma 4.5]{EmeryRatcliffeTschantz}).

\begin{maincor} \label{ThmLengthsnotclassical}
Let $n$ be odd and let $\Gamma$ be an arithmetic lattice in $\mathrm{Isom}(\H^n)$ commensurable with some $\Or(q,\Z)$ for an admissible quadratic form $q$ over $\Q$. Suppose $q$ has reduced determinant $-D$. Let $m=\frac{n+1}{2}$. The number of distinct lengths $\leq L$ of hyperbolic isometries which are elements of classical arithmetic lattices in the commensurability class of $\Gamma$ is bounded above by
\[
    \frac{\kappa(m,D)}{\pi\sqrt{D}}e^{2(m-1)L}L + O(e^{2(m-1)L})
\]
where $\kappa(m,D)$ is a constant depending only on $D$ and $m$.
\end{maincor}

The study of multiplicities in the length spectrum of arithmetic hyperbolic manifolds is an active area of research. The exponential growth of multiplicity averages for surfaces follows from work by Luo and Sarnak \cite{LS94}. A similar result has recently been obtained for semi-arithmetic hyperbolic surfaces \cite{BCDP26}. For $n\geq 3$, this was initiated in \cite{BLMT22} for the even case and in \cite{G24} for the odd case. We apply our methods to improve lower bounds for the strong exponential growth of average of multiplicities in the geodesic length spectrum of non-compact arithmetic orbifolds. 

\begin{mainthm} \label{ThmMeanMult}
Let $X$ be a classical non-compact arithmetic hyperbolic orbifold of odd dimension $n>4$ of the first type. Then, the multiplicity average of the primitive length spectrum of $X$ has growth rate at least 
\[ c\cdot\frac{e^{\frac{(n-1)\ell}{2}}}{\ell^2}, \hspace{3mm}\ell\to\infty \]
where $c$ is a positive constant.
\end{mainthm}

This article is organized as follows. We give the necessary background on quadratic forms, hyperbolic space, arithmetic lattices, and Salem numbers in \Cref{s:background}. In \Cref{s:Diophantine} we prove the following statement which we were unable to find in the literature, and which may be of independent interest. While more general results exist (cf. \cite{efthymios}), we include an elementary proof that is in the spirit of counting Pythagorean triples.

\newtheorem*{laterthm}{\Cref{th:diop. sol.}}
\begin{laterthm}
    Let $D$ be a square-free positive integer and $t$ be the number of distinct prime factors of $D$. Denote the number of integer solutions to $A^2+DB^2=C^2$ with $|C|\leq X$ by
    \[ P_D(X)=\#\{(A,B,C)\in \Z^{3}\mid A^2+DB^2=C^2, \mbox{  and  } |C|\leq X\}.
    \]
    Then
    \[ P_D(X) = \begin{cases}
    2^{t}\frac{8X}{\pi\sqrt{D}} + O(\sqrt{X}\log X)   &\text{ if $D$ is odd}\\
    2^{t}\frac{6X}{\pi\sqrt{D}} + O(\sqrt{X}\log X)   &\text{ if $D$ is even} .
    \end{cases} \]
\end{laterthm}

In \Cref{s:Salem count} we prove \Cref{ThmSalemCount} and give explicit constants $\kappa(D,m)$ in \Cref{th:Count with constants}. We then apply \Cref{ThmSalemCount} to prove \Cref{ThmSalemProportion} and deduce \Cref{ThmLengths} in \Cref{s:count lengths}. Finally, in \Cref{s:mean mult}, we prove \Cref{ThmMeanMult}.

\section*{Acknowledgements}
This paper arose from the Number Theory in the Americas II workshop. We would like to thank the organizers and the funding agencies that supported the workshop (Banff International Research Station, Casa Matem\'atica Oaxaca, IMU-CDC, JINVARIANT, Clay Mathematics Institute, Journal of Number Theory) for providing us with the opportunity to work together. We would also like to thank Kiran Kedlaya,
Moubariz Garaev, Harald Helfgott, Frits Beukers, and Jan-Hendrik Evertse for helpful discussions on some of the background material related to this paper. We are extremely grateful to Frits Beukers for suggesting a superior approach to the Diophantine counting problem that ultimately led to \Cref{th:diop. sol.}. 

The first author was partially supported by NSF grants DMS-2243188 and DMS-2441034. The second author was partially supported by Universal CNPq Grant 408834/2023-4. The third author was supported by the SECIHTI Grant 169113, \emph{Estancias Posdoctorales por México 2023-Modalidad Académica}, at CIMAT. The fourth author was partially supported by the Max Planck Institute for Mathematics and the Centre de Recherches Mathématiques during her sabbatical.

\section{Background}\label{s:background}

\subsection{Quadratic forms}
Let $K$ be a field. A \emph{quadratic form} is a homogeneous polynomial in $n$ variables of degree 2 with coefficients in $K$. It can be written as 
\[ q(x)=\sum_{i=1}^n \sum_{j=i}^n a_{ij}x_i x_j = x^t S_q x \] where $S_q$ is the symmetric matrix with diagonal entries $(S_q)_{ii}=a_{ii}$ and off-diagonal entries $(S_q)_{ij}=\frac{1}{2}a_{ij}$. The \emph{rank} of $q$ is the number of variables $n$. The \emph{determinant} of $q$ is $\det(q)=\det(S_q)$. We will assume that $q$ is nondegenerate, that is, $S_q$ is invertible.

We say that a quadratic form $q$ represents $a\in K$ if there is a vector $v\in K^n$ satisfying $q(v)=a$. A quadratic form is \emph{isotropic} if it represents 0 by a nonzero vector and it is \emph{anisotropic} otherwise.

Two quadratic forms $q$ and $q'$ over $K$ are said to be $K$-equivalent if there exists a matrix $M\in\GL(n,K)$ such that $M^t S_q M= S_{q'}$. Any quadratic form over $K$ is $K$-equivalent to a diagonal quadratic form $q'$, which means $q'(x)=\sum_{i=1}^n a_ix_i^2$. We write $q'=\langle a_1,\dots,a_n \rangle$.

Let $V=K^n$. Then the \emph{quadratic space} $(V,q)$ has isometry group given by the orthogonal group 
\[ \Or(q)=\Or(q,K)=\{ A \in \GL(n,K) \:|\: A^tS_qA=S_q \}. \]

We will focus on quadratic forms $q$ defined over $K=\Q$.
The signature of $q$ is the pair $(n_p, n_n)$ where $n_p$ is the number of positive eigenvalues of $S_q$ and $n_n$ is the number of negative eigenvalues. 
Since $q$ is defined over $\Q$, we may scale $q$ by a positive integer so that the coefficients and the determinant lie in $\Z$. Note that scaling does not change the rank, the signature, nor the isometry group.
Considering $(-1)$ as a ``prime'' for our purposes, let  $\det(q)=p_1^{\epsilon_1}\cdots p_k^{\epsilon_k}$ be the prime decomposition for $\det(q)$. The \emph{reduced determinant} of $q$ is the square-free integer $\det_r(q)=p_1^{\epsilon'_1}\cdots p_k^{\epsilon'_k}$ with $\epsilon'_i=0$ when $\epsilon_i$ is even and $\epsilon'_i=1$ when $\epsilon_i$ is odd.

\subsection{Hyperbolic space} \label{ss:hyperbolic space}
The hyperbolic $n$-space $\H^n$ is the simply connected Riemannian $n$-manifold in which the sectional curvature is constant $-1$. We will consider the hyperboloid model of hyperbolic space defined as follows.

Let $n \geq 2$ and $q=\langle 1,\dots , 1,-1\rangle$ be the standard quadratic form with signature $(n,1)$. The bilinear form $B$ in $\R^{n+1}$ associated with the quadratic form $q$ is given by
\[ {\displaystyle B(x,y) = \frac{q(x+y)-q(x)-q(y)}{2}}. \]
The set 
\[ \mathcal{S}=\{(x_1,\ldots,x_{n+1}) \in \R^{n+1}; q(x_1,\ldots,x_{n+1})=-1 \}\]
forms a two-sheeted hyperboloid. We denote the positive sheet by 
\[ \mathcal{S}^+=\{(x_1,\ldots,x_{n+1}) \in \mathcal{S}; x_{n+1}>0 \}.\] 
The \textit{hyperboloid model} of hyperbolic $n$-space is obtained by identifying $\H^n$ with the $n$-manifold $\mathcal{S}^+$ together with the Riemannian metric given by $d(x,y)=\arccos\left(B(x,y)\right)$. 

The orthogonal group $\Or(q,\R)$ preserves the hyperboloid $\mathcal{S}$. We identify the isometry group 
\[ \mathrm{Isom}(\H^n)=\mathrm{PO}(q,\R)=\mathrm{PO}(n,1)\cong\SO(q,\R)=\SO(n,1)\]
with the isometry group of $\H^n$ in this model. It has two connected components. The identity component consists of orientation-preserving isometries and the other component consists of orientation-reversing isometries. 
We remark that one may construct isometric hyperboloid models for $\H^n$ from any quadratic form $q$ with signature $(n,1)$ defined over a real field.

An element $\gamma\in \mathrm{Isom}(\H^n)$ is called \emph{hyperbolic} if there is a unique geodesic $L\in\H^n$, called the \emph{axis} of $\gamma$, on which $\gamma$ acts as a translation by a positive distance $\ell(\gamma)$, called the \emph{translation length} of $\gamma$. In this case, $\gamma$ has exactly $n-1$ eigenvalues of norm one and two real eigenvalues not on the unit circle, $\lambda=e^{\ell(\gamma)}>1$ and $\lambda^{-1}=e^{-\ell(\gamma)}<1$ (see \cite{ChenGreenberg}). 

A \emph{lattice} is a discrete subgroup $\Gamma<\mathrm{Isom}(\H^n)$ such that the quotient $\H^n/\Gamma$ has finite volume. If $\H^n/\Gamma$ is compact, we say that $\Gamma$ is \textit{cocompact}. Whenever $\Gamma$ is a lattice, the quotient $\H^n/\Gamma$ is a complete finite-volume hyperbolic orbifold. In this case, most elements of $\Gamma$ are hyperbolic and the axes of hyperbolic elements map to closed geodesics in the quotient. The quotient $\H^n/\Gamma$ will be a manifold whenever $\Gamma$ is torsion free. 

Two lattices $\Gamma_1$ and $\Gamma_2$ in $\mathrm{Isom}(\H^n)$ are \emph{commensurable} if there exists $g\in\mathrm{Isom}(\H^n)$ such that $\Gamma_1\cap g\Gamma_2g^{-1}$ has finite index in both $\Gamma_1$ and $g\Gamma_2g^{-1}$. 

\subsection{Arithmetic lattices}\label{ss:arithmetic lattices}
Let $K$ be a totally real number field with ring of integers $\mathcal{O}_K$. Consider a quadratic form $q:\R^{n+1} \rightarrow \R$ with signature $(n,1)$ defined over $K$ with respect to the identity embedding $K\hookrightarrow \R$. We say that such a form $q$ is \emph{admissible} if $q^\sigma$ has signature $(n,0)$ with respect to any other non-identity embedding $\sigma:K\hookrightarrow \R$ (in other words, $q^\sigma$ is positive definite). This condition on the Galois conjugates of $q$ guarantees that the group
\[ \SO(q,\mathcal{O}_K) = \{ A \in \GL(n+1,\mathcal{O}_K); A^tS_qA=S_q\}\]
is a discrete subgroup of $\SO(q,\R)$. Furthermore, $\SO(q,\mathcal{O}_K)$ is a lattice in $\mathrm{Isom}(\H^n)$ in the sense that it is a discrete subgroup of finite co-volume; that is, the quotient $\H^n/\SO(q,\mathcal{O}_K)$ has finite volume with respect to the hyperbolic metric.
Any lattice $\Gamma<\mathrm{Isom}(\H^n)$ commensurable with some $\SO(q,\mathcal{O}_K)$ for an admissible $q$ defined over a totally real number field $k$ is called \emph{arithmetic of the first type} (also known as \emph{arithmetic of simplest type}). The quotient orbifold $\H^n/\Gamma$ is also called \emph{arithmetic of the first type}.

An arithmetic lattice of the first type $\Gamma$ is called \emph{classical} if it is contained in $\SO(q,K)$ for an admissible $q$ defined over a totally real number field $k$. It turns out that every arithmetic lattice in $\mathrm{Isom}(\H^n)$ for $n$ even is classical, and necessarily of the first type (see, e.g., \cite[Lemma 4.2]{EmeryRatcliffeTschantz}). When $n$ is odd, there are arithmetic lattices which are not of the first type; however, if $\Gamma$ is arithmetic of the first type, then $\Gamma^{(2)}$, the subgroup generated by the squares of the elements in $\Gamma$, is classical (see \cite[Lemma 4.5]{EmeryRatcliffeTschantz}).

We focus on arithmetic lattices of the first type defined over $\Q$. All non-cocompact arithmetic lattices are of the first type and defined over $\Q$. The quotient hyperbolic orbifold $\H^n/\SO(q,\mathcal{O}_K)$ is compact whenever $q$ is anisotropic, and cusped whenever $q$ is isotropic. In dimension $2$, there is a unique commensurability class of non-cocompact arithmetic lattices which contains, via the isomorphism $\mathrm{Isom}^+(\H^2)\cong\PSL(2,\R)$, the modular group $\PSL(2,\Z)$. All other arithmetic lattices of the first type defined over $\Q$ in dimension $2$ are cocompact.
In dimension $3$, there are infinitely many commensurability classes of non-cocompact arithmetic lattices each of which contains, via the isomorphism $\mathrm{Isom}^+(\H^3)\cong\PSL(2,\C)$, a Bianchi group $\PSL(2,R_D)$ where $R_D$ is the ring of integers in $\Q(\sqrt{-D})$. There are also infinitely many commensurability classes of cocompact arithmetic lattices in dimension $3$ which are of the first type and defined over $\Q$.
Whenever $n\geq 4$, it follows from Meyer's theorem that the quadratic forms over $\Q$ with signature $(n,1)$ are always isotropic. Thus in dimensions at least $4$, arithmetic lattices of the first type defined over $\Q$ are exactly those which are non-cocompact.

\subsection{Salem numbers and lengths of geodesics}
For any non-constant polynomial
\[ p(x)=a_dx^d+a_{d-1}x^{d-1}+\dots + a_1x+a_0= a_{d} \prod_{i=1}^d (x-\alpha_i ),\]
we define the \textit{Mahler measure of $p(x)$} as
\[ M(p) = |a_{d}| \prod_{i=1}^d \mathrm{max}(|\alpha_i|,1).\]
One can show via Jensen's formula that this is equivalent to
\[
M(p)=\exp\left( \int_0^1 \log |p(e^{2\pi i t})|\; dt \right),
\]
 which can be viewed as the geometric mean of $|p(z)|$ over all values of $z$ lying on the unit circle. One can immediately observe that $M(p) = 1$ precisely when $p(x) = x^k \displaystyle\prod_i \Phi_{n_i}(x)$, where $\Phi_n$ denotes the $n$-th cyclotomic polynomial. There is a lower bound on the Mahler measure in terms of the coefficients (see \cite[\S 2]{Mahler})
 \begin{equation}\label{eq:coeffbound}
     |a_k|\leq \binom{d}{k}M(p) . 
 \end{equation} 

In 1933, Lehmer asked whether the Mahler measure reaches a minimal value when considered over all minimal polynomials of algebraic integers that are not roots of unity. This became known as Lehmer's conjecture, which can be stated as follows:

\begin{conjecture}[Lehmer's conjecture]
There exists a real number $m > 1$ such that $M(p) \geq m$ for all non-cyclotomic polynomials $p$. 
\end{conjecture}

The smallest Mahler measure (greater than $1$) that is currently known is attained by Lehmer's polynomial,
\[ L(x)=x^{10}+x^9-x^7-x^6-x^5-x^4-x^3+x+1,\]
which has Mahler measure $M(L) \approx 1.176280818$. By plotting the roots of Lehmer's polynomial, one can observe that the real roots are $\lambda = 1.176280818....$ and $\lambda^{-1} = 0.850137...$. There are also eight complex roots of $L(x)$ that lie along the unit circle. This motivates a definition: 

\begin{definition} A \emph{Salem number} is a real algebraic integer $\lambda > 1$ such that all of its Galois conjugates except $\lambda^{-1}$ have norm $1$. A polynomial $f\in\Z[x]$ is a \emph{Salem polynomial} if it is the minimal polynomial of a Salem number.
\end{definition}

In 2019, Emery, Ratcliffe and Tschantz established a more direct relationship between Salem numbers and lengths of geodesics in hyperbolic orbifolds which are arithmetic of the first type. Whenever $\gamma$ is a hyperbolic element in a classical arithmetic lattice in $\mathrm{Isom}(\H^{n})$, $\lambda=e^{\ell(\gamma)}$ is a Salem number (see \Cref{th:ERT 1.1}). Conversely, we have:

\begin{theorem}\cite[Theorem 6.3]{EmeryRatcliffeTschantz}
Let $\lambda$ be a Salem number, let K be a subfield of $\Q(\lambda+\lambda^{-1})$, and let $n\geq 2$ such that $\deg_K(\lambda)\leq n+1$. Then there exist a classical arithmetic lattice $\Gamma\subset\Isom(\H^n)$ of the first type over $K$ and an orientation preserving hyperbolic element $\gamma\in\Gamma$ such that $\lambda=e^{\ell(\gamma)}$.
\end{theorem}

Due to this connection, it is possible to reformulate Lehmer's conjecture in terms of geodesic lengths. 

\begin{conjecture}[Short Geodesic Conjecture]
There exists a positive universal lower bound on the length of geodesics in an arithmetic hyperbolic orbifold of the first type.
\end{conjecture}

We say that a Salem number $\lambda$ is \emph{realized} in $\Gamma<\mathrm{Isom}(\H^{n})$ if there exists a hyperbolic element $\gamma\in \Gamma$ such that $\lambda=e^{\ell(\gamma)}$, alternatively if $\lambda$ is the real eigenvalue of $\gamma$ satisfying $\lambda>1$.

\section{Counting Diophantine solutions} \label{s:Diophantine}

Let $m$ be a positive integer and $D$ a square-free positive integer. We are interested in counting the number of Salem polynomials $f$ of degree $2m$ for which $f(1)f(-1)\equiv -D$ in $\Q^{\times}/\Q^{\times 2}$ with Mahler measure $\leq Q$. Any such $f$ produces a solution $(A,B,C)\in\mathbb{Z}^3$ of the degree $2$ Diophantine equation 
\begin{equation}\label{eq:diophantine eq}
 A^2+DB^2=C^2  
\end{equation}
for the following reasons. Since $D$ is square-free, we can assume that $f(1)f(-1)=-Dk^2$ with $k\in\Z$. Now, if $f(x)=\sum^{2m}_{k=0} c_k x^k \in\Z[x] \text{ with }c_k=c_{2m-k} \text{ and } c_0=c_{2m}=1$ take 

\[A=\sum_{k\\ \text{ even }} c_k,\hspace{1cm} C=\sum_{k \text{ odd}} c_k,\hspace{1cm} B=k.\]

The goal of this section is to prove the following.

\begin{theorem}\label{th:diop. sol.}
    Let $D$ be a square-free positive integer and $t$ be the number of distinct prime factors of $D$. Denote the number of integer solutions to $A^2+DB^2=C^2$ with $\gcd(A,C)=1$ and $|C|\leq X$ by
    \[ P_D(X)=\#\{(A,B,C)\in \Z^{3}\mid A^2+DB^2=C^2, \, \gcd(A,C)=1,  \text{  and  } |C|\leq X\}.
    \]
    Then
    \[ P_D(X) = \begin{cases}
    2^{t}\frac{8X}{\pi\sqrt{D}} + O(\sqrt{X}\log X)   &\text{ if $D$ is odd},\\
    2^{t}\frac{6X}{\pi\sqrt{D}} + O(\sqrt{X}\log X)   &\text{ if $D$ is even} .
    \end{cases} \]
\end{theorem}

We will begin by stating some useful lemmas about the number of lattice points in convex sets.

Let $\mathcal{C} \subset\R^2$ be a bounded, convex, centrally-symmetric region.
As a bounded set, there exists an $R\in \R_{>0}$ such that any point $(x,y)\in\mathcal{C}$ satisfies $-R\leq x\leq R$ and $-R\leq y\leq R$.

Define $N(\alpha, \mathcal{C}) = \#\{\alpha\mathcal{C} \cap \mathbb{Z}^2\}$ and $M(\alpha, \mathcal{C}) = \#\{\alpha \mathcal{C} \cap (2\mathbb{Z}+1)^2\}$ for all $\alpha \in \R_{>0}$.

\begin{lemma}\label{lem:convexcount} For all $\alpha$, we have 
\begin{align*}
    N(\alpha, \mathcal{C}) &= \alpha^2 \mathrm{vol}(\mathcal{C}) + O(\alpha), \\ M(\alpha, \mathcal{C}) &= \frac{1}{4} \alpha^2 \mathrm{vol}(\mathcal{C}) + O(\alpha).\
\end{align*}    
\end{lemma}

\begin{proof} The key is to adapt Gauss' classical lattice point counting argument, which says that we can express 
\[ \#\{\mathrm{Lattice \ points \ in \ a \ circle\} = \mathrm{Area} + O(\mathrm{Perimeter}).} \] 
In our case, we relate the lattice point count of $\alpha \mathcal{C}$ to the volume of the bounded region $\mathcal{C}$ and the size of its boundary. Since $\alpha \mathcal{C}$ is a dilation of $\mathcal{C}$ by a factor $\alpha$ then the boundary of $\alpha\mathcal{C}$ will be a linear function of $\alpha$. The volume of $\alpha\mathcal{C}$ is $\alpha^2 \mathrm{vol}(\mathcal{C})$. Thus, we obtain 
\[ N(\alpha, \mathcal{C}) = \alpha^2 \mathrm{vol}(\mathcal{C}) + O(\alpha).\]
When computing $M(\alpha, \mathcal{C})$, we only consider odd lattice points, which shrinks the count by a factor of $4$ since both coordinates in $\mathbb{Z}^2$ must be odd. 
\end{proof}

We are interested in restricting the counts of $N(\alpha, \mathcal{C})$ and $M(\alpha, \mathcal{C})$ to only the primitive points, that is, points $(A,B)\in \Z^2$ for which $\gcd(A, B) = 1$. Thus, let $N'(\alpha, \mathcal{C})$ and $M'(\alpha, \mathcal{C})$ denote the corresponding counts for the primitive points. Then we can derive the following result from \Cref{lem:convexcount}. 

\begin{lemma}\label{lem:primitive}
With $N'$ and $M'$ defined as above, we have
\begin{align*}
    N'(\alpha, \mathcal{C}) &= \frac{6 \alpha^2}{\pi^2} \mathrm{vol}(\mathcal{C}) + O(\alpha \log \alpha),\\ 
M'(\alpha, \mathcal{C}) &= \frac{2\alpha^2}{\pi^2} \mathrm{vol}(\mathcal{C}) + O(\alpha \log \alpha). 
\end{align*}
\end{lemma}

\begin{proof} The proof follows from \cite[Theorem 459]{hardywright} with our $\alpha$ playing the role of $1/\rho$. We also have our $N(\alpha,\mathcal{C})$ and $N'(\alpha,\mathcal{C})$ playing the roles of $g(\rho)+1$ and $f(\rho)+1$ respectively. We will bound the difference between $N'(\alpha, \mathcal{C})$ and $\frac{6 \alpha^2}{\pi^2}$, showing that it is $O(\alpha \log \alpha)$. According to \cite[24.10.8]{hardywright},
\begin{equation} \label{eq:HW}
    N'(\alpha,\mathcal{C})-1-\frac{\alpha^2 \mathrm{vol}(\mathcal{C})}{\zeta(2)} 
= \alpha^2 \sum_{m=1}^\infty \frac{\mu(m)}{m^2}\left( \frac{m^2}{\alpha^2} \left(N({\alpha/m},\mathcal{C})-1\right)-\mathrm{vol}(\mathcal{C})\right). 
\end{equation} 

Note that if $\frac{\alpha}{m}<\frac{1}{R}$, then $N(\alpha/m,\mathcal{C})=1$ since this will only count the origin. Thus, the right-hand side of the above expression becomes
\begin{align*} &\sum_{m = 1}^{\lfloor\alpha R\rfloor} \mu(m)\left(N\left(\frac{\alpha}{m}, \mathcal{C}\right) - 1\right) - \alpha^2 \mathrm{vol}(\mathcal{C}) \sum_{m = 1}^\infty \frac{\mu(m)}{m^2}\\ & = \alpha^2 \mathrm{vol}(\mathcal{C}) \sum_{m=1}^{\lfloor \alpha R \rfloor} \frac{\mu(m)}{m^2} + O\left(\alpha \sum_{m = 1}^{\lfloor \alpha R \rfloor} \frac{\mu(m)}{m}\right) - \alpha^2 \mathrm{vol}(\mathcal{C}) \sum_{m=1}^\infty \frac{\mu(m)}{m^2} 
\\ & =  O\left(\alpha \sum_{m = 1}^{\lfloor \alpha R \rfloor} \frac{\mu(m)}{m}\right) - \alpha^2 \mathrm{vol}(\mathcal{C}) \sum_{m=\lceil \alpha R \rceil}^\infty \frac{\mu(m)}{m^2} ,
\end{align*} where the second line follows from applying \Cref{lem:convexcount} to the first term.

For the sum inside the Big O, one can use the classical result (which follows from partial summation) that
\[ \sum_{m = 1}^{\lfloor \alpha R \rfloor} \frac{\mu(m)}{m} \leq \sum_{m = 1}^{\lfloor \alpha R \rfloor} \frac{1}{m} \leq \log \alpha R + \gamma + O\left(\frac{1}{\alpha R}\right),\] where $\gamma$ is the Euler-Mascheroni constant. 

Therefore, the expression in \eqref{eq:HW} is bounded by 
\[ O\left(\alpha \log (\alpha R) + \alpha \gamma + \frac{1}{R}\right) -\alpha^2 \mathrm{vol}(\mathcal{C}) \sum_{m = \lceil \alpha R \rceil}^\infty \frac{\mu(m)}{m^2} \]
where the last sum can be bounded above by $\frac{1}{\alpha R - 1},$ so the final subtracted term has magnitude $O(\alpha)$. Since $R$ is a fixed constant, depending only on $\mathcal{C}$, we can absorb it into the Big O constant, yielding the claimed error term of $O(\alpha \log \alpha).$

To obtain the count for $M'(\alpha, \mathcal{C})$ from the count for $N'(\alpha, \mathcal{C})$, we use the fact that primitive points cannot have two even coordinates.
If we write our lattice points as $(A, B)$ then the probability that at least one of $A$ or $B$ is not divisible by $p$ is given by $1 - \frac{1}{p^2}$. So, looking for the number of points where neither $A$ nor $B$ is divisible by $p = 2$ simply amounts to removing $p = 2$ from the Euler product and considering that $\frac{1}{4}$ of all lattice points will have the combination with $A$ and $B$ both odd. Thus, we have \[ M'(\alpha, \mathcal{C})= \frac{1}{4}\frac{1}{\left(1-\frac{1}{2^2}\right)} N'(\alpha, \mathcal{C})=\frac{1}{3} N'(\alpha, \mathcal{C}).  \]
\end{proof}

Note that a sharp form of the estimate $N'(\alpha, \mathcal{C})$ appears in \cite{huxleynowack} but we will not need anything that strong for our purposes. 

We are now ready to prove \Cref{th:diop. sol.}.

\begin{proof}[Proof of \Cref{th:diop. sol.}]
Let $D$ be positive and square-free. We will count integer solutions to the Diophantine equation $A^2 + DB^2 = C^2$ where $\gcd(A, C) = 1$ and $C \leq X$. We begin with the classical method of drawing an ellipse $A^2 + DB^2 = 1$ and drawing a reference point along the ellipse (for example, $(1, 0)$), and then drawing a line from this reference point to a rational point $r$ along the $y$-axis. The point where the line intersects the ellipse will also be a rational solution to $A^2 + DB^2 = 1$. A line through both $(1,0)$ and $r$ is given by the equation $B = (1-A)r$. By substituting $B = (1-A)r$ into the Diophantine equation and solving for $A$ and $B$, we obtain $A = \frac{-1 + Dr^2}{1 + Dr^2}$ and $B = \frac{2r}{1+Dr^2}$. Setting $r=V/U$ yields the parametrization 
    \begin{equation} \label{eq:all sols param}
    A= S\frac{-U^2+DV^2}{T},\quad B=S\frac{2UV}{T}, \text{ and } C=S\frac{U^2+DV^2}{T}, \end{equation}
where $S,U,V$ are integers satisfying $\gcd(U,V)=1$, and $T$ is the greatest common divisor of the numerators.

Suppose $(A,B,C)$ is a positive primitive solution written as in \eqref{eq:all sols param}. We may then assume that $S=1$, $U,V\geq0$, and $U^2\leq DV^2$. Let $D_1=\gcd\{U^2+DV^2,UV\}$. Since $\gcd(U,V)=1$, it must be that $D_1$ divides both $D$ and $U$. There is some $D_2>0$ and $u\geq0$ such that $D=D_1D_2$ and $U=D_1u$. It follows that any positive primitive solution can be written as 
\begin{equation} \label{eq:pos prim param}
     (A,B,C)=\frac{1}{\tau}(-D_1u^2+D_2v^2,2uv,D_1u^2+D_2v^2) 
\end{equation}
where $D_1 > 0$, $D_2 >0$, $D=D_1D_2$, $u \geq 0$, $v \geq 0$, $\gcd(u, v) = 1$, $D_1u^2\leq D_2v^2$, and $\tau=1$ when $B$ is even or otherwise $\tau=2$ and $u,v$ are odd when $B$ is odd.

We will now count the number of positive primitive solutions. We will consider four cases, depending on the parity of $D$ and on the relationship between $A$ and $C$. Notice that, if $\gcd(A, C) = 1$, then $\gcd((C-A), (C+A)) = 1$ or $2$. 

\vspace{10pt}
\noindent \emph{Case I:} $D$ is even and $A\equiv C\bmod2$. 

Then $B$ must be even. 
By \eqref{eq:pos prim param}, there are unique integers $D_1 > 0, D_2 >0$, $u \geq 0$, and $v \geq 0$ with $\gcd(u, v) = 1$ for which 
\[ A = D_2 v^2 - D_1 u^2, \hspace{0.2 in} B = 2uv, \hspace{0.2 in} C = D_1u^2 + D_2v^2.\]
We will show that every choice of such $D_1, D_2, u, v$ produces a positive primitive solution.

The congruence condition on $A$ and $C$ implies that $\gcd(C-A, C+A) = 2$. Therefore, there is no odd prime that divides both $A$ and $C$. If $A$ and $C$ are both even then $D_1 u^2$ and $D_2 v^2$ are both odd, which contradicts the fact that $D = D_1 D_2$ is even. As a result, for each choice of $D_1$ and $D_2$, we will count the elements in the set 
\[ \{(u, v) \in (\mathbb{Z}_{\geq 0})^2 : \gcd(u, v) = 1, D_1u^2 + D_2v^2 \leq X, D_1u^2 \leq D_2 v^2\}.\]
This count is precisely $N'(\sqrt{X}, \mathcal{C})$ where 
\[ \mathcal{C} = \{(x,y) \in (\mathbb{R}_{\geq 0})^2 : D_1x^2 + D_2y^2 \leq 1, D_1x^2 \leq D_2 y^2\}. \]
Observe that 
\[ \mathrm{vol}(\mathcal{C}) = \frac{\pi}{8 \sqrt{D_1 D_2}} = \frac{\pi}{8 \sqrt{D}}. \]
From our formula for $N'(X, \mathcal{C})$, we obtain 
\[ N'(\sqrt{X}, \mathcal{C}) \sim \frac{6}{\pi^2} \times \frac{\pi}{8\sqrt{D}} \times (\sqrt{X})^2 \sim \frac{3X}{4 \pi \sqrt{D}}. \]
As we run over all choices of $D_1$, our count becomes 
\[ 2^t \frac{3X}{4 \pi\sqrt{D}} + O(\sqrt{X} \log X), \]
where the error term comes from the formula for $N'(\sqrt{X}, \mathcal{C})$.

\vspace{10pt}
\noindent \emph{Case II:} $D$ is even and $A\not\equiv C\bmod2$. 

This case does not occur, since it would contradict $C^2-A^2=D B^2$.

\vspace{10pt}
\noindent \emph{Case III:} $D$ is odd and $A\equiv C\bmod2$.

Since $A \equiv C\bmod2$, then $B$ must be even. By \eqref{eq:pos prim param}, there are unique integers $D_1 > 0, D_2 >0$, $u \geq 0$, and $v \geq 0$ with $\gcd(u, v) = 1$ for which 
\[ A = D_2 v^2 - D_1 u^2, \hspace{0.2 in} B = 2uv, \hspace{0.2 in} C = D_1u^2 + D_2v^2.\]

Each choice of such $D_1, D_2, u, v$ produces a solution to $A^2+DB^2=C^2$. However, not all choices produce a primitive solution. The greatest common divisor may be $2$, which occurs if both $u$ and $v$ are odd. 

Thus, for each choice of $D_1, D_2$, we count the elements in the set 
\[ \{(u,v) \in (\mathbb{Z}_{\geq 0})^2 \mid \gcd(u, v) = 1, D_1u^2 + D_2v^2 \leq X, D_1u^2 \leq D_2v^2\} \]
and subtract off the number of elements in the set 
\[ \{(u,v) \in (2\mathbb{Z}_{\geq 0} + 1)^2 \mid \gcd(u, v) = 1, D_1u^2 + D_2v^2 \leq X, D_1u^2 \leq D_2v^2\}.\]
This is precisely given by 
\[ N'(\sqrt{X}, \mathcal{C}) - M'(\sqrt{X}, \mathcal{C}),\]
where $\mathcal{C}$ is the same as in \textit{Case I}. Using our formulae for $N'$ and $M'$, we obtain 
\[ N'(\sqrt{X}, \mathcal{C}) \sim \frac{3X}{4 \pi \sqrt{D}},\]
as in \textit{Case I}, minus 
\[ M'(\sqrt{X}, \mathcal{C}) \sim \frac{2}{\pi^2} \times \frac{\pi}{8 \sqrt{D}} \times (\sqrt{X})^2 \sim \frac{X}{4 \pi\sqrt{D}}. \]
This difference is asymptotically $\frac{X}{2 \pi \sqrt{D}}$. As we run over all choices of $D_1$, we obtain a contribution of $2^t \frac{X}{2\pi \sqrt{D}} + O(\sqrt{X} \log X).$\\

\vspace{10pt}
\noindent \emph{Case IV:} $D$ is odd and $A\not\equiv C\bmod2$.

Then $B$ must be odd and \(\gcd(C+A, C-A) = 1\). By \eqref{eq:pos prim param}, there are unique integers $D_1 > 0, D_2 >0$, $u \geq 0$, and $v \geq 0$ with $\gcd(u, v) = 1$ and both $u,v$ odd for which 
\[ A = \frac{1}{2}(D_2 v^2 - D_1 u^2), \hspace{0.2 in} B = uv, \hspace{0.2 in} C = \frac{1}{2}(D_1u^2 + D_2v^2).\]
Every choice of such $D_1, D_2, u, v$ produces a positive primitive solution to $A^2+DB^2=C^2$. Indeed, if an odd prime $p$ is a common divisor of $A$ and $C$, then $p$ also divides $A\pm C$, which is not possible since $\gcd(C+A, C-A) = 1$.

We also have
\[ 2C = D_1 u^2 + D_2 v^2 \leq 2X.\]
Then, for each choice of $D_1, D_2$, we count the elements in the set 
\[ \{(u, v)\in (2\mathbb{Z}_{\geq 0} + 1)^2: \gcd(u, v)=1, D_1u^2 + D_2v^2 \leq 2X, D_1u^2 \leq D_2v^2\}.\]
This is precisely 
\[ M'(\sqrt{2X}, \mathcal{C}) \sim \frac{2}{\pi^2} \times \frac{\pi}{8\sqrt{D}} \times (\sqrt{2X})^2 \sim \frac{X}{2 \pi \sqrt{D}}.\]
As we run over all choices of $D_1$, we obtain a total count of $2^t \frac{X}{2\pi\sqrt{D}} + O(\sqrt{X} \log X).$

\vspace{10pt}
Collecting the information from all four cases, we see that the contribution from even values of $D$ yields $2^t \frac{3X}{4\pi \sqrt{D}} + O(\sqrt{X} \log X)$ solutions to the Diophantine equation $A^2 + DB^2 = C^2$. On the other hand, the contribution from odd values of $D$ yields $2^t \frac{X}{\pi \sqrt{D}} + O(\sqrt{X} \log X)$ solutions. Observe that the above arguments only count positive solutions. Allowing for all $2^3$ combinations of $\pm$ solutions yields the result in \Cref{th:diop. sol.}.
\end{proof}

\begin{remark} The above argument produces our main term. However, we can replace the error term above with a much sharper version due to Huxley \cite{huxley}, which allows us to improve the error term to $O(X^{\frac{131}{416}})$. 
\end{remark}

\begin{remark} We remark that when $D = 1$ and $t = 0$ our \Cref{th:diop. sol.} coincides with the result for Pythagorean triples given by Fricker in \cite{fricker}. 
\end{remark}

\section{Counting Salem numbers} \label{s:Salem count}

Let $D>0$ be a square-free positive integer. Denote by $\mathcal{F}_{m,D}(Q)$ the set of polynomials of degree $2m$ of the form
\begin{equation}\label{eq:palyndromic}
    f(x)=\sum^{2m}_{k=0} c_k x^k \in\Z[x] \text{ with }c_k=c_{2m-k} \text{ and } c_0=c_{2m}=1 
\end{equation}
which satisfy the conditions:
\begin{enumerate}[(i)]\label{eq:conditions}
    \item $f$ is irreducible over $\Q$, 
    \item $f(1)f(-1)=-D k^2$ where $k\in \Q$, and
    \item $f$ has distinct roots $\lambda,\lambda^{-1},\mu_1,\mu_1^{-1},\dots,\mu_{m-1},\mu_{m-1}^{-1}$ where $\lambda\in\R$ with $1<\lambda\leq Q$ and each $\mu_i\notin\R$ with $|\mu_i|=1$.
\end{enumerate}
Since $D$ is square-free, condition (ii) is equivalent to having $f(1)f(-1)=-D k^2$ where $k\in \Z$.

Denote by $\mathcal{S}_m(Q)$ the number of polynomials of degree $2m$ with form as in \eqref{eq:palyndromic} which only satisfy conditions (i) and (iii). G{\"o}tze and Gusakova showed the following.

\begin{theorem}\cite[Theorem 1.1]{GG20} \label{thm:GG Salem count}
    With $\mathcal{S}_m(Q)$ defined as above, we have \[ \#\mathcal{S}_m(Q) = \omega_{m}Q^m+O(Q^{m-1}) \]
    where
    \[ \omega_m= \frac{2^{m(m-1)}}{m} \prod^{m-2}_{k=0} \frac{k!^2}{(2k+1)!}. \]
\end{theorem}

Condition (ii) can be rewritten as $A^2+DB^2=C^2$ by setting $B=k$ and noting that $f(1)f(-1)=A^2-C^2$ where $A$ is the sum of the coefficients with even powers of $x$ and $C$ is the sum of the coefficients with odd powers of $x$:
\begin{equation}\label{eq:defAB}
 A=\sum_{\substack{0\leq k\leq 2m \\ \text{even}}} c_k \text{ and } C=\sum_{\substack{1\leq k\leq 2m-1 \\ \text{odd}}} c_k.    
\end{equation}

\begin{lemma} \label{lem:ACbounds}
With $A, C, Q$ defined as above, we have $|C| \leq 2^{2m-1} Q $ and \\ $ |A |\leq (2^{2m-1}-2) Q +2$.
\end{lemma}

\begin{proof}
    It follows from the Binomial Theorem that $\sum_{k\text{ odd}}\binom{2m}{k}=2^{2m-1}$. Then using \eqref{eq:coeffbound} together with Condition (iii), which implies $M(p)\leq Q$, we obtain
    \[ |C|= |\sum_{k\text{ odd}} c_k| \leq \sum_{k\text{ odd}} |c_k| 
    \leq \sum_{k\text{ odd}} \binom{2m}{k} Q = 2^{2m-1} Q. \]
    The second assertion follows a similar argument after setting $c_0=c_{2m}=1$.
\end{proof}

The following well-known result is stated as in \cite{MR1705753} (see also \cite{Netto,MR271277}).

\begin{lemma}\label{th:partitionfn}
    Let $\alpha=\{a_1,\cdots,a_n\}$ be a set of $n$ relatively prime positive integers. Let the partition function $p_\alpha(N)$ denote the number of non-negative integer solutions to the linear equation $ a_1x_1+a_2x_2+\cdots+a_nx_n=N$. Then
    \[ p_\alpha(N) = \frac{N^{n-1}}{(a_1\cdot a_2\cdots a_n)(n-1)!} + O(N^{n-2}). \]
\end{lemma}

We will use \Cref{th:partitionfn} to obtain the following result:

\begin{lemma} \label{lem:ck count}
Given $A$ and $C$ satisfying \Cref{lem:ACbounds}, the number of integer solutions for $\{c_0,\dots,c_{2m}\}$ satisfying \Cref{eq:coeffbound,eq:palyndromic,eq:defAB} is bounded above by
\[
    \kappa_0(m) Q^{m-2} + O(Q^{m-3}),
\]
where 
\[ \kappa_0(m)=
    \begin{cases}
         \frac{2^{(m-1)(m-2)}(2^{2m}-4)^{\frac{m}{2}-1}}{\left( (\frac{m}{2}-1)! \right)^2} & \text{ if $m$ is even} \\
         \frac{2^{(2m-1)(\frac{m-1}{2})}(2^{2m-1}-2)^{\frac{m-3}{2}}}{\left( (\frac{m-1}{2})! \right)^2} & \text{ if $m$ is odd} .
    \end{cases} 
\]
\end{lemma}

\begin{proof}
For fixed $A$ and $C$, we count the number of integral solutions $\{c_0,\dots,c_{2m}\}$ satisfying $c_k\geq -\binom{2m}{k}\floor{Q}$ (see \eqref{eq:coeffbound}) by applying \Cref{th:partitionfn} to count the number of non-negative solutions after the substitutions $c'_k=c_k+\binom{2m}{k}\floor{Q}$. Note that we ignore the upper bound $c_k\leq \binom{2m}{k}\floor{Q}$.

Suppose first that $m$ is even. Then $C=2(c_1+c_3+\dots+c_{m-1})$, so $C$ is even as well. For a fixed value of $C$, the number of integral solutions to this linear equation which has $\frac{m}{2}$ variables is
\[ \frac{\left( \frac{1}{2}C+2^{2m-2}\floor{Q} \right)^{\frac{m}{2}-1}}{(\frac{m}{2}-1)!} + O\left( \left( \frac{1}{2}C+2^{2m-2}\floor{Q} \right)^{\frac{m}{2}-2} \right). \]

We also have that $A-2=2c_2+2c_4+\cdots+2c_{m-2}+c_m$. For a fixed value of $A$, the number of integral solutions to this linear equation which has $\frac{m}{2}$ variables is
\[ \frac{\left( A-2+(2^{2m-1}-2)\floor{Q} \right)^{\frac{m}{2}-1}}{2^{\frac{m}{2}-1}(\frac{m}{2}-1)!} + O\left( \left( A-2+(2^{2m-1}-2)\floor{Q} \right)^{\frac{m}{2}-2} \right). \]

Combining this with the bounds in \Cref{lem:ACbounds}, we obtain for $m$ even that the number of solutions is bounded above by 
\[
    \frac{2^{(m-1)(m-2)}(2^{2m}-4)^{\frac{m}{2}-1}}{\left( (\frac{m}{2}-1)! \right)^2} Q^{m-2} + O(Q^{m-3}).
\]

Suppose now that $m$ is odd. Then $C=2c_1+2c_3+\dots+2c_{m-2}+c_m$. For a fixed value of $C$, the number of integral solutions to this linear equation with $\frac{m+1}{2}$ variables is 
\[ \frac{\left( C+2^{2m-1}\floor{Q} \right)^{\frac{m-1}{2}}}{2^{\frac{m-1}{2}}(\frac{m-1}{2})!} + O\left( \left( C+2^{2m-1}\floor{Q} \right)^{\frac{m-3}{2}} \right). \]

We also have that $A-2=2(c_2+c_4+\cdots+c_{m-1})$, so $A$ is necessarily even. For a fixed value of $A$, the number of integral solutions to this linear equation which has $\frac{m-1}{2}$ variables is
\[ \frac{\left( \frac{A-2}{2}+(2^{2m-2}-1)\floor{Q} \right)^{\frac{m-3}{2}}}{(\frac{m-1}{2})!} + O\left( \left( \frac{A-2}{2}+(2^{2m-2}-1)\floor{Q} \right)^{\frac{m-5}{2}} \right). \]

As before, we conclude for $m$ odd that the number of solutions is bounded above by 
\[
    \frac{2^{(2m-1)(\frac{m-1}{2})}(2^{2m-1}-2)^{\frac{m-3}{2}}}{\left( (\frac{m-1}{2})! \right)^2} Q^{m-2} + O(Q^{m-3}).
\]
\end{proof}

We are now able to prove \Cref{ThmSalemCount} with explicit constants.

\begin{theorem}\label{th:Count with constants}
With the notation as above, we have    
     \[
    \# \mathcal{F}_{m,D}(Q) \leq \frac{\kappa(m,D)}{\pi\sqrt{D}}Q^{m-1}\log Q + O(Q^{m-1})
    \]
    where 
    \[ \kappa(m,D)=
    \begin{cases}
        2^{t+2m}\cdot\kappa_0(m) & \text{ if $D$ is odd} \\
        3\cdot2^{t+2m-2}\cdot\kappa_0(m) & \text{ if $D$ is even}
    \end{cases} 
    \]
    and $\kappa_0(m)$ is as in \Cref{lem:ck count}.
\end{theorem}

\begin{proof}[Proof of \Cref{ThmSalemCount}]
\Cref{th:diop. sol.} allows us to count the number of integer solutions for the triple $(A,B,C)$ satisfying $A^2+DB^2=C^2$ with $|C| \leq 2^{2m-1} Q$. However, we are interested in counting the integer solutions only for the tuple $(A,C)$, so we ignore the sign of $B$ by dividing by $2$. Moreover, we only want to count the subset of solutions where $C$ is even whenever $m$ is even and $A$ is even whenever $m$ is odd (see the proof of \Cref{lem:ck count}), so we again divide by $2$. This results in the following count for admissible integral solutions $(A,C)$ with $\gcd(A,C)=1$:
\begin{equation}\label{coefficient}
\begin{array}{ll}
    \frac{2^{t+2m}}{\pi\sqrt{D}}Q + O(\sqrt{Q}\log Q) & \text{if $D$ is odd} \\
    \frac{3\cdot2^{t+2m-2}}{\pi\sqrt{D}}Q + O(\sqrt{Q}\log Q) & \text{if $D$ is even.} 
\end{array}
\end{equation}
However, we are interested in \textit{all} integral solutions and not only these primitive solutions. To compute this quantity, we observe that:

\begin{align}&\#\{(A, C) \in \mathbb{Z}^2: A^2 + DB^2 = C^2 \ \mathrm{with} \ |C| \leq Q\} \notag \\ &= 1 + \sum_{N \leq Q}\#\{(A, C) \in \mathbb{Z}^2: A^2 + DB^2 = C^2, \gcd(A, C) = 1, \ \mathrm{and} \ |C| \leq Q/N\} \notag \\ &= 1 + \sum_{N \leq Q} \left(\frac{2^J}{\pi\sqrt{D}} \frac{Q}{N} + O\left(\sqrt{\frac{Q}{N}} \log\frac{Q}{N}\right) \right) \label{eq:applthm3.1} \\ &=1 + \frac{2^J}{\pi \sqrt{D}}Q \sum_{N \leq Q} \frac{1}{N} + \sum_{N \leq Q}O\left(\sqrt{\frac{Q}{N}}\log\frac{Q}{N}\right) \label{eq:expanded}, \end{align} where \eqref{eq:applthm3.1} follows from applying Theorem \ref{th:diop. sol.}, and $J$ denotes the numerator of the constant appearing in Theorem \ref{th:diop. sol.} (note that the value of $J$ changes depending on whether $D$ is even or odd).

 First we handle the main term in \eqref{eq:expanded}. By partial summation (see, for example, \cite[Theorem 422]{hardywright68}),  we have $\sum_{N\leq Q} \frac{1}{N} = \log Q + \gamma + O(\frac{1}{Q})$, where $\gamma = 0.57721...$ is the Euler-Mascheroni constant. This can be simplified to $\sum_{N\leq Q} \frac{1}{N} = \log Q + O(1).$

 Next, we handle the error term in \eqref{eq:expanded}. From the last line in \cite[Section 2]{fricker}, we have
 \[ \sum_{N \leq Q} \frac{1}{\sqrt{N}} \log \frac{Q}{N} = O(\sqrt{Q}). \] Thus, we can bound the error term in \eqref{eq:expanded} as follows:
 
\[ \sum_{N \leq Q} \sqrt{\frac{Q}{N}} \log \frac{Q}{N} = \sqrt{Q} \sum_{N \leq Q} \frac{1}{\sqrt{N}} \log \frac{Q}{N} = \sqrt{Q} \cdot O(\sqrt{Q})= O(Q). \]

 Therefore, we can simplify \eqref{eq:expanded} to $\frac{2^J}{\pi \sqrt{D}} Q \log Q + O(Q).$ By plugging in the values of $J$ that appear in Theorem \ref{th:diop. sol.}, we obtain the following count for all admissible integral solutions $(A,C)$:
\begin{equation}\label{eq:count ThA}
\begin{array}{ll}
    \frac{2^{t+2m}}{\pi\sqrt{D}}Q\log Q + O(\sqrt{Q}\log Q) & \text{if $D$ is odd} \\
    \frac{3\cdot2^{t+2m-2}}{\pi\sqrt{D}}Q\log Q + O(\sqrt{Q}\log Q) & \text{if $D$ is even.} 
\end{array}
\end{equation}

Let $\mathcal{F}'_{m,D}(Q)$ denote the set of polynomials of the form in \eqref{eq:palyndromic} satisfying conditions (ii) and (iii) (but not necessarily (i)). The size of $\mathcal{F}'_{m,D}(Q)$ is bounded by the product of \eqref{eq:count ThA} with the bound in \Cref{lem:ck count}; that is,
    \[
    \# \mathcal{F}'_{m,D}(Q) \leq \frac{\kappa(m,D)}{\pi\sqrt{D}}Q^{m-1}\log Q + O(Q^{m-1}), 
    \]
with $\kappa(m,D)$ as in the statement of the theorem.

Now if some polynomial $f\in \mathcal{F}'_{m,D}(Q)$ is not irreducible, it can be written as a product $f=g\cdot h$ with $g,h\in\Z[x]$. Since $f$ satisfies condition (iii), we may assume that $g$ contains the root $\lambda$. It must then be the case that $g\in \mathcal{F}'_{m_1,D}(Q)$ is a polynomial of degree $2m_1$ and $h$ is a product of cyclotomic polynomials of degree $2m_2$, where $m_1+m_2=m$ and $m_1,m_2>0$. The number of such polynomials $h$ is bounded above by some constant $\iota_0(m_2)$. It follows that the subset $\mathcal{F}^{red}_{m,D}(Q)\subset \mathcal{F}'_{m,D}(Q)$ of reducible polynomials satisfies  
    \[
    \# \mathcal{F}^{red}_{m,D}(Q) \leq \sum_{i=1}^{m-1}\iota_0(i)\cdot \#\mathcal{F}'_{m-i}(Q)
    \leq \sum_{i=1}^{m-1}\iota_0(i)\cdot \frac{\kappa(i,D)}{\pi\sqrt{D}}Q^{m-i-1}\log Q
    \leq \iota(m,D) Q^{m-2}\log Q
    \]
for some constant $\iota(m,D)$ depending only on $m$ and $D$. Since $\# \mathcal{F}_{m,D}(Q)=\# \mathcal{F}'_{m,D}(Q) - \# \mathcal{F}^{red}_{m,D}(Q)$, the theorem now follows by absorbing $\# \mathcal{F}^{red}_{m,D}(Q)$ into the error term.
\end{proof}

\section{Counting lengths of closed geodesics in commensurability classes}\label{s:count lengths}
Let $q$ be an admissible quadratic form with signature $(n,1)$ over $\Q$, $n$ odd. The goal of this section is to count the Salem numbers realized in the commensurability class of $\SO(q,\Z)$.

\begin{lemma}\label{lem:QtoZ}
    Let $T$ be an isometry in $\SO(q,\Q)$ where $q$ is an admissible quadratic form over $\Q$. Suppose also that its characteristic polynomial $p_T(x)$ has integer coefficients. Then a conjugate of $T$ is contained in a classical arithmetic lattice commensurable with $\SO(q,\Z)$.
\end{lemma}

\begin{proof}
    The rank of $q$ is $n+1$, which is also the degree of $p_T\in\Z[x]$. The matrix powers $T,T^2,\dots,T^{n}$ all have entries in $\Q$. Let $\ell$ denote the least common multiple of the denominators appearing in the entries of these $n$ rational matrices.
    By the Cayley-Hamilton Theorem, $T$ satisfies its characteristic polynomial, so for any $k\geq 0$, we have 
    \[ T^{k+n+1}=-\left( T^k +c_1 T^{k+1}+\cdots +c_{r-1}T^{k+n} \right) . \]
    
    The subgroup of $\Q^{n+1}$ generated by the $T^k\Z^{n+1}$ for all $k\geq 0$ lies between the subgroups $\Z^{n+1}$ and $(\frac{1}{\ell}\Z)^{n+1}$ and is thus itself a free abelian group of rank $n+1$. We can therefore choose a change of basis matrix $g\in\GL(n+1,\Q)$ such that $g^{-1}Tg$ has entries in $\Z$.
    
    If $S_q$ is the symmetric matrix associated to the quadratic form $q$, let $q'$ be the quadratic form associated to the symmetric matrix $S_{q'}=g^t S_q g$. The result follows from checking that $g^{-1} T g\in \SO(q',\Z)$ and the fact that the automorphism groups of $\Q$-equivalent rational quadratic forms are commensurable, so $\SO(q,\Z)$ and $\SO(q',\Z)$ are commensurable.
\end{proof}

\begin{proof}[Proof of \Cref{ThmSalemProportion}]
    The Salem numbers which are realizable in the commensurability class of $\SO(q,\Z)$ have even degree $2m\leq n+1$. Let $\mathcal{R}_q(Q)$ denote the set of realizable Salem numbers with $\lambda<Q$. By \cite{Montesinos-Amilibia}, the reduced determinant of a rational quadratic form of even rank is a commensurability invariant. By \cite[Corollary 9.3]{Bayer-Fluckinger}, a necessary condition for the existence of an isometry in $\mathrm{O}(q,\Q)$ with minimal polynomial $f$ is that $\det(q)\equiv f(1)f(-1)$ in $\Q^\times/\Q^{\times2}$. Since $f(0)=1$, such an isometry will be in $\SO(q,\Q)$, and hence by \Cref{lem:QtoZ}, it will be in the commensurability class of $\SO(q,\Z)$. Therefore we have
    \[  \mathcal{R}_q(Q) \subseteq \left(\sqcup_{i=1}^{m-1} \mathcal{S}_i(Q) \right)\sqcup \mathcal{F}_{m,D}(Q), \]
    with $\mathcal{F}_{m,D}(Q)$ as in \Cref{s:Salem count}. By \Cref{thm:GG Salem count} and \Cref{th:Count with constants},  
    \begin{align*}
        \# \mathcal{R}_q(Q) & \leq \# \mathcal{F}_{m,D}(Q)+ \sum_{i=1}^{m-1} \# \mathcal{S}_i(Q) 
        \\ &\leq  \frac{\kappa(m,D)}{\pi\sqrt{D}}Q^{m-1}\log Q + O(Q^{m-1}),
    \end{align*} 
    since for $i<m$, $\# \mathcal{S}_i(Q)$ is absorbed into $O(Q^{m-1})$.
\end{proof}

We can now deduce \Cref{ThmLengths} and \Cref{ThmLengthsnotclassical}.

\begin{proof}[Proof of \Cref{ThmLengths} and \Cref{ThmLengthsnotclassical}]
    Suppose first that $\Gamma$ is classical. Recall that the length of a hyperbolic element $\gamma\in\Gamma$ satisfies $\ell(\gamma)=\log\lambda$ where $\lambda$ is the unique real eigenvalue of $\gamma$ such that $\lambda>1$. By \Cref{th:ERT 1.1}, $\lambda$ is a Salem number. Then $\ell(\gamma)\leq L$ if and only $\lambda\leq e^L$. 
    \Cref{ThmLengths} follows from \Cref{ThmSalemProportion} by replacing $Q$ with $e^L$.

    Now suppose $\Gamma$ is not classical. Given a hyperbolic element $\gamma\in\Gamma$, its square $\gamma^2$ is in the classical finite-index subgroup $\Gamma^{(2)}\subset\Gamma$ generated by squares of elements of $\Gamma$. Then the unique real eigenvalue $\lambda>1$ of $\gamma^2$ is a Salem number. Here $2\ell(\gamma)=\ell(\gamma^2)=\log\lambda$. Then $\ell(\gamma)\leq L$ if and only if
    $\lambda\leq e^{2L}$.
    \Cref{ThmLengthsnotclassical} similarly follows from \Cref{ThmSalemProportion} by replacing $Q$ with $e^{2L}$.
\end{proof}

\section{Bounds on multiplicity of the length spectrum} \label{s:mean mult}

Let $X=\H^n/\Gamma$ be a hyperbolic orbifold, and let $L(X)=\{\ell_1<\ell_2<\cdots\}$ be its primitive length spectrum without multiplicities. For a given $\ell_k\in L(X)$, the multiplicity $g(\ell_k)$ is defined as the number of homotopy classes of primitive closed geodesics in $X$ with length $\ell_k$.

It is known that the length spectrum counting multiplicities contains important geometric and analytic information about $X$. For example, for $X$ a compact hyperbolic manifold, Kelmer shows in \cite{K14} that the length spectrum with multiplicities determines the Laplace spectrum of $X$.

Determining the numbers $\ell_k, g(\ell_k)$ for $k\in\N$ has been shown to be a difficult problem about which very little is known. Numerical calculations have been performed for the octagonal Fuchsian group \cite{ABS91}, which already shows large fluctuations in the multiplicities. For modular and Bianchi orbifolds, the multiplicity relates to the class number of indefinite binary quadratic forms (see \cite{S82,S83}). The determination of such class numbers is a wide open problem in algebraic number theory.

For arithmetic hyperbolic orbifolds, a lower bound for the average of the multiplicities can be obtained by counting Salem numbers. This study was initiated in \cite{BLMT22} for even $n$, and in \cite{G24} for odd $n$. Both of these papers used results about the so-called \emph{square-rootable} Salem numbers. Let \[N(\ell)=\sum_{\ell_k\leq\ell}g(\ell_k),\hspace{1cm} \widehat{N}(\ell)=\sum_{\ell_k\leq\ell }  1, \]
and
\[G(\ell)=\frac{N(\ell)}{\widehat{N}(\ell)} . \]

Then $N(\ell)$ counts the number of primitive closed geodesics in $X$ with length $\leq \ell$, $\widehat{N}(\ell)$ counts the lengths in $L(X)$ up to $\ell$, and $G(\ell)$ counts the average of the multiplicities $g(\ell_k)$ with $\ell_k\leq\ell$. Our techniques allow us to give a lower bound on $G(\ell)$ for an arithmetic hyperbolic orbifold of odd dimension $n>4$, when $\ell\to\infty$.

\begin{proof}[Proof of \Cref{ThmMeanMult}]
By the Primitive Geodesic Theorem for non-compact lattices proved by Gangolli and Warner \cite[Prop. 5.4]{GW80}, $N(\ell)\sim e^{(n-1)\ell}/(n-1)\ell$. So, for $\ell$ sufficiently large, \[N(\ell)\geq \frac{1}{2(n-1)\ell}e^{(n-1)\ell}.\]

Since $X$ is non-compact, it is defined over $\Q$. Since $X$ is classical, each length $\leq\ell$ of a closed geodesic corresponds to a Salem number $\leq e^\ell$ of degree $\leq n+1$ (see \Cref{th:ERT 1.1}). By \Cref{ThmLengths} there is a constant $r>0$, depending only on $n$ and on the commensurability class of $X$, such that \[\widehat{N}(\ell)\leq r\cdot \ell e^{(n-1)\ell/2},\]
and then
\[G(\ell)\geq \frac{1}{2r(n-1)\ell^2}e^{\frac{(n-1)\ell}{2}}, \hspace{1cm}\ell\to\infty.\] 
\end{proof}

It is worth mentioning that our definition of `multiplicity average' bounded in \Cref{ThmMeanMult} coincides with the definition of \emph{mean multiplicity} in \cite{BLMT22} and \cite{G24}. However, in this article, we call it the multiplicity average in order to distinguish our definition from the \emph{mean multiplicity} $\langle g_p(\ell)\rangle$ studied by Bolte \cite{B93} and Marklof \cite{M96}. 
Unfortunately, we are unable to compute the mean multiplicity in the sense of Bolte and Marklof, since the asymptotics for $\widehat{N}(\ell)$ are not available to us at the present time.

\bibliographystyle{alpha}
\bibliography{biblio}

@article {Bayer-Fluckinger,
    AUTHOR = {Bayer-Fluckiger, Eva},
     TITLE = {Isometries of quadratic spaces},
   JOURNAL = {J. Eur. Math. Soc. (JEMS)},
  FJOURNAL = {Journal of the European Mathematical Society (JEMS)},
    VOLUME = {17},
      YEAR = {2015},
    NUMBER = {7},
     PAGES = {1629--1656},
      ISSN = {1435-9855,1435-9863},
   MRCLASS = {11E12 (11E04)},
  MRNUMBER = {3361725},
MRREVIEWER = {Andrzej\ S\l adek},
       DOI = {10.4171/JEMS/541},
       URL = {https://doi.org/10.4171/JEMS/541},
}

@incollection {ChenGreenberg,
    AUTHOR = {Chen, S. S. and Greenberg, L.},
     TITLE = {Hyperbolic spaces},
 BOOKTITLE = {Contributions to analysis (a collection of papers dedicated to
              {L}ipman {B}ers)},
     PAGES = {49--87},
 PUBLISHER = {Academic Press, New York-London},
      YEAR = {1974},
   MRCLASS = {53C35 (20G20 32N05)},
  MRNUMBER = {377765},
MRREVIEWER = {Rolf\ Sulanke},
}

@article {EmeryRatcliffeTschantz,
    AUTHOR = {Emery, Vincent and Ratcliffe, John G. and Tschantz, Steven T.},
     TITLE = {Salem numbers and arithmetic hyperbolic groups},
   JOURNAL = {Trans. Amer. Math. Soc.},
  FJOURNAL = {Transactions of the American Mathematical Society},
    VOLUME = {372},
      YEAR = {2019},
    NUMBER = {1},
     PAGES = {329--355},
      ISSN = {0002-9947,1088-6850},
   MRCLASS = {11R06 (11E10 11F06 20H10 30F40)},
  MRNUMBER = {3968771},
MRREVIEWER = {James\ McKee},
       DOI = {10.1090/tran/7655},
       URL = {https://doi.org/10.1090/tran/7655},
}

@article {fricker,
    AUTHOR = {Fricker, Fran\c{c}ois},
     TITLE = {{\"U}ber die {V}erteilung der pythagoreischen {Z}ahlentripel},
   JOURNAL = {Arch. Math. (Basel)},
  FJOURNAL = {Archiv der Mathematik},
    VOLUME = {28},
      YEAR = {1977},
    NUMBER = {5},
     PAGES = {491--494},
      ISSN = {0003-889X,1420-8938},
   MRCLASS = {10H25},
  MRNUMBER = {444590},
MRREVIEWER = {H.\ W.\ Brinkmann},
       DOI = {10.1007/BF01223955},
       URL = {https://doi.org/10.1007/BF01223955},
}

@article {huxley,
    AUTHOR = {Huxley, M. N.},
     TITLE = {Exponential sums and lattice points. {II}},
   JOURNAL = {Proc. London Math. Soc. (3)},
  FJOURNAL = {Proceedings of the London Mathematical Society. Third Series},
    VOLUME = {66},
      YEAR = {1993},
    NUMBER = {2},
     PAGES = {279--301},
      ISSN = {0024-6115,1460-244X},
   MRCLASS = {11P21 (11L07 11P55)},
  MRNUMBER = {1199067},
MRREVIEWER = {Wolfgang\ M\"uller},
       DOI = {10.1112/plms/s3-66.2.279},
       URL = {https://doi.org/10.1112/plms/s3-66.2.279},
}

@article{GG20,
    AUTHOR = {G\"otze, Friedrich and Gusakova, Anna},
     TITLE = {On the distribution of {S}alem numbers},
   JOURNAL = {J. Number Theory},
  FJOURNAL = {Journal of Number Theory},
    VOLUME = {216},
      YEAR = {2020},
     PAGES = {192--215},
      ISSN = {0022-314X,1096-1658},
   MRCLASS = {11R06 (11P21 60B20)},
  MRNUMBER = {4130080},
MRREVIEWER = {Wataru\ Takeda},
       DOI = {10.1016/j.jnt.2020.02.012},
       URL = {https://doi.org/10.1016/j.jnt.2020.02.012},
}

@article{BLMT22,
    AUTHOR = {Belolipetsky, Mikhail and Lal\'in, Matilde and Murillo, Plinio
              G. P. and Thompson, Lola},
     TITLE = {Counting {S}alem numbers of arithmetic hyperbolic 3-orbifolds},
   JOURNAL = {Bull. Braz. Math. Soc. (N.S.)},
  FJOURNAL = {Bulletin of the Brazilian Mathematical Society. New Series.
              Boletim da Sociedade Brasileira de Matem\'atica},
    VOLUME = {53},
      YEAR = {2022},
    NUMBER = {2},
     PAGES = {553--569},
      ISSN = {1678-7544,1678-7714},
   MRCLASS = {11R06 (20H05 20H10 57R18)},
  MRNUMBER = {4418784},
MRREVIEWER = {James\ McKee},
       DOI = {10.1007/s00574-021-00270-9},
       URL = {https://doi.org/10.1007/s00574-021-00270-9},
}

@article {Mahler,
    AUTHOR = {Mahler, K.},
     TITLE = {On some inequalities for polynomials in several variables},
   JOURNAL = {J. London Math. Soc.},
  FJOURNAL = {The Journal of the London Mathematical Society},
    VOLUME = {37},
      YEAR = {1962},
     PAGES = {341--344},
      ISSN = {0024-6107,1469-7750},
   MRCLASS = {10.30 (32.05)},
  MRNUMBER = {138593},
MRREVIEWER = {L.\ Mirsky},
       DOI = {10.1112/jlms/s1-37.1.341},
       URL = {https://doi.org/10.1112/jlms/s1-37.1.341},
}

@article{GW80,
    AUTHOR = {Gangolli, Ramesh and Warner, Garth},
     TITLE = {Zeta functions of {S}elberg's type for some noncompact
              quotients of symmetric spaces of rank one},
   JOURNAL = {Nagoya Math. J.},
  FJOURNAL = {Nagoya Mathematical Journal},
    VOLUME = {78},
      YEAR = {1980},
     PAGES = {1--44},
      ISSN = {0027-7630,2152-6842},
   MRCLASS = {58G10 (10D24 10D40)},
  MRNUMBER = {571435},
MRREVIEWER = {P.\ C.\ Trombi},
       URL = {http://projecteuclid.org/euclid.nmj/1118786087},
}

@article {MR1705753,
    AUTHOR = {Nathanson, Melvyn B.},
     TITLE = {Partitions with parts in a finite set},
   JOURNAL = {Proc. Amer. Math. Soc.},
  FJOURNAL = {Proceedings of the American Mathematical Society},
    VOLUME = {128},
      YEAR = {2000},
    NUMBER = {5},
     PAGES = {1269--1273},
      ISSN = {0002-9939,1088-6826},
   MRCLASS = {11P82 (05A17)},
  MRNUMBER = {1705753},
MRREVIEWER = {Yifan\ Yang},
       DOI = {10.1090/S0002-9939-00-05606-9},
       URL = {https://doi.org/10.1090/S0002-9939-00-05606-9},
}

@book {MR271277,
    AUTHOR = {P\'olya, Georg and Szeg\"o, G\'abor},
     TITLE = {Aufgaben und {L}ehrs\"atze aus der {A}nalysis. {B}and {I}:
              {R}eihen, {I}ntegralrechnung, {F}unktionentheorie},
    SERIES = {Heidelberger Taschenb\"ucher [Heidelberg Paperbacks]},
    VOLUME = {Band 73},
      NOTE = {Vierte Auflage},
 PUBLISHER = {Springer-Verlag, Berlin-New York},
      YEAR = {1970},
     PAGES = {xvi+338},
   MRCLASS = {26.00},
  MRNUMBER = {271277},
}

@book {Netto,
    AUTHOR = {Netto, Eugen},
     TITLE = {Lehrbuch der {C}ombinatorik},
 PUBLISHER = {Chelsea Publishing Co., New York},
      YEAR = {1958},
     PAGES = {viii+341},
   MRCLASS = {05.00},
  MRNUMBER = {95126},
}

@article {Montesinos-Amilibia,
    AUTHOR = {Montesinos-Amilibia, Jos\'e{} Mar\'ia},
     TITLE = {On integral quadratic forms having commensurable groups of
              automorphisms},
   JOURNAL = {Hiroshima Math. J.},
  FJOURNAL = {Hiroshima Mathematical Journal},
    VOLUME = {43},
      YEAR = {2013},
    NUMBER = {3},
     PAGES = {371--411},
      ISSN = {0018-2079,2758-9641},
   MRCLASS = {11E04 (11E20 57M25 57M50 57M60)},
  MRNUMBER = {3161323},
MRREVIEWER = {Andrzej\ S\l adek},
       URL = {http://projecteuclid.org/euclid.hmj/1389102581},
}

@article{K14,
    AUTHOR = {Kelmer, Dubi},
     TITLE = {A refinement of strong multiplicity one for spectra of
              hyperbolic manifolds},
   JOURNAL = {Trans. Amer. Math. Soc.},
  FJOURNAL = {Transactions of the American Mathematical Society},
    VOLUME = {366},
      YEAR = {2014},
    NUMBER = {11},
     PAGES = {5925--5961},
      ISSN = {0002-9947,1088-6850},
   MRCLASS = {11F72 (11F70 11M36 22E45)},
  MRNUMBER = {3256189},
MRREVIEWER = {Anton\ Deitmar},
       DOI = {10.1090/S0002-9947-2014-06102-3},
       URL = {https://doi.org/10.1090/S0002-9947-2014-06102-3},
}

@article{ABS91,
    AUTHOR = {Aurich, R. and Bogomolny, E. B. and Steiner, F.},
     TITLE = {Periodic orbits on the regular hyperbolic octagon},
   JOURNAL = {Phys. D},
  FJOURNAL = {Physica D. Nonlinear Phenomena},
    VOLUME = {48},
      YEAR = {1991},
    NUMBER = {1},
     PAGES = {91--101},
      ISSN = {0167-2789,1872-8022},
   MRCLASS = {58F17 (81Q20 81Q50)},
  MRNUMBER = {1098656},
MRREVIEWER = {Donal\ Hurley},
       DOI = {10.1016/0167-2789(91)90053-C},
       URL = {https://doi.org/10.1016/0167-2789(91)90053-C},
}

@article{S82,
    AUTHOR = {Sarnak, Peter},
     TITLE = {Class numbers of indefinite binary quadratic forms},
   JOURNAL = {J. Number Theory},
  FJOURNAL = {Journal of Number Theory},
    VOLUME = {15},
      YEAR = {1982},
    NUMBER = {2},
     PAGES = {229--247},
      ISSN = {0022-314X,1096-1658},
   MRCLASS = {10C07 (10C15)},
  MRNUMBER = {675187},
MRREVIEWER = {O.\ M.\ Fomenko},
       DOI = {10.1016/0022-314X(82)90028-2},
       URL = {https://doi.org/10.1016/0022-314X(82)90028-2},
}

@article{S83,
    AUTHOR = {Sarnak, P.},
     TITLE = {The arithmetic and geometry of some hyperbolic
              three-manifolds},
   JOURNAL = {Acta Math.},
  FJOURNAL = {Acta Mathematica},
    VOLUME = {151},
      YEAR = {1983},
    NUMBER = {3-4},
     PAGES = {253--295},
      ISSN = {0001-5962,1871-2509},
   MRCLASS = {11F72 (11E41 58G25)},
  MRNUMBER = {723012},
MRREVIEWER = {P.\ G.\ Zograf},
       DOI = {10.1007/BF02393209},
       URL = {https://doi.org/10.1007/BF02393209},
}

@article{G24,
    AUTHOR = {Grebennikov, Alexandr},
     TITLE = {Multiplicities in the length spectrum and growth rate of
              {S}alem numbers},
   JOURNAL = {Bull. Braz. Math. Soc. (N.S.)},
  FJOURNAL = {Bulletin of the Brazilian Mathematical Society. New Series.
              Boletim da Sociedade Brasileira de Matem\'atica},
    VOLUME = {55},
      YEAR = {2024},
    NUMBER = {2},
     PAGES = {Paper No. 25, 21},
      ISSN = {1678-7544,1678-7714},
   MRCLASS = {11K16 (11F06 20H10 57R18)},
  MRNUMBER = {4747505},
MRREVIEWER = {James\ McKee},
       DOI = {10.1007/s00574-024-00398-4},
       URL = {https://doi.org/10.1007/s00574-024-00398-4},
}

@article{B93,
    AUTHOR = {Bolte, Jens},
     TITLE = {Periodic orbits in arithmetical chaos on hyperbolic surfaces},
   JOURNAL = {Nonlinearity},
  FJOURNAL = {Nonlinearity},
    VOLUME = {6},
      YEAR = {1993},
    NUMBER = {6},
     PAGES = {935--951},
      ISSN = {0951-7715,1361-6544},
   MRCLASS = {58F19 (11F72 58F20)},
  MRNUMBER = {1251251},
MRREVIEWER = {Shin-ya\ Koyama},
       DOI = {10.1088/0951-7715/6/6/006},
       URL = {https://doi.org/10.1088/0951-7715/6/6/006},
}

@article{M96,
    AUTHOR = {Marklof, J.},
     TITLE = {On multiplicities in length spectra of arithmetic hyperbolic
              three-orbifolds},
   JOURNAL = {Nonlinearity},
  FJOURNAL = {Nonlinearity},
    VOLUME = {9},
      YEAR = {1996},
    NUMBER = {2},
     PAGES = {517--536},
      ISSN = {0951-7715,1361-6544},
   MRCLASS = {11F72 (57M50 58G25)},
  MRNUMBER = {1384490},
MRREVIEWER = {Kerry\ N.\ Jones},
       DOI = {10.1088/0951-7715/9/2/014},
       URL = {https://doi.org/10.1088/0951-7715/9/2/014},
}

@article{efthymios,
    AUTHOR = {Sofos, Efthymios},
     TITLE = {Uniformly counting rational points on conics},
   JOURNAL = {Acta Arith.},
  FJOURNAL = {Acta Arithmetica},
    VOLUME = {166},
      YEAR = {2014},
    NUMBER = {1},
     PAGES = {1--14},
      ISSN = {0065-1036,1730-6264},
   MRCLASS = {11D45 (14G05)},
  MRNUMBER = {3273493},
MRREVIEWER = {Daniel\ Loughran},
       DOI = {10.4064/aa166-1-1},
       URL = {https://doi.org/10.4064/aa166-1-1},
}

@article{huxleynowack,
    AUTHOR = {Huxley, Martin N. and Nowak, Werner Georg},
     TITLE = {Primitive lattice points in convex planar domains},
   JOURNAL = {Acta Arith.},
  FJOURNAL = {Acta Arithmetica},
    VOLUME = {76},
      YEAR = {1996},
    NUMBER = {3},
     PAGES = {271--283},
      ISSN = {0065-1036,1730-6264},
   MRCLASS = {11P21 (11M26 11M41)},
  MRNUMBER = {1397317},
MRREVIEWER = {Douglas\ Hensley},
       DOI = {10.4064/aa-76-3-271-283},
       URL = {https://doi.org/10.4064/aa-76-3-271-283},
}

@book {hardywright,
    AUTHOR = {Hardy, G. H. and Wright, E. M.},
     TITLE = {An introduction to the theory of numbers},
   EDITION = {Sixth},
      NOTE = {Revised by D. R. Heath-Brown and J. H. Silverman,
              With a foreword by Andrew Wiles},
 PUBLISHER = {Oxford University Press, Oxford},
      YEAR = {2008},
     PAGES = {xxii+621},
      ISBN = {978-0-19-921986-5},
   MRCLASS = {11-01},
  MRNUMBER = {2445243},
}

@book {hardywright68,
    AUTHOR = {Hardy, G. H. and Wright, E. M.},
     TITLE = {An introduction to the theory of numbers},
   EDITION = {Fourth},
 PUBLISHER = {Oxford University Press, Oxford},
      YEAR = {1968},
     PAGES = {xxii+621},
      ISBN = {978-0-19-921986-5},
   MRCLASS = {11-01},
  MRNUMBER = {2445243},
}

@article{LS94,
  title={Number variance for arithmetic hyperbolic surfaces},
  author={Luo, Wenzhi and Sarnak, Peter},
  journal={Communications in mathematical physics},
  volume={161},
  number={2},
  pages={419--432},
  year={1994},
  publisher={Springer}
}

@article{BCDP26,
  title={On Multiplicities in Length Spectra of Semi-Arithmetic Hyperbolic Surfaces},
  author={Belolipetsky, Mikhail and Cosac, Gregory and D{\'o}ria, Cayo and Teixeira Paula, Gisele},
  journal={Communications in Mathematical Physics},
  volume={407},
  number={4},
  pages={76},
  year={2026},
  publisher={Springer}
}

\end{document}